\documentclass[reqno, 12pt]{amsart}
\usepackage[utf8]{inputenc}
\usepackage{amssymb,amsmath,amsfonts,amsthm,calrsfs,mathtools}
\usepackage{color}
\usepackage[english]{babel}
\usepackage[T1]{fontenc}
\usepackage{latexsym}
\usepackage{enumitem}
\usepackage{tabularx}
\usepackage{booktabs}
\usepackage[a4paper,top=3cm,bottom=2cm,left=3cm,right=3cm,marginparwidth=1.75cm]{geometry}
\usepackage[colorinlistoftodos]{todonotes}
\usepackage[colorlinks=true, allcolors=red]{hyperref}
\usepackage{tikz}
\usepackage{amsmath,amssymb}
\usepackage{caption}

\theoremstyle{plain}
\newtheorem{theorem}{Theorem}
\newtheorem{lemma}{Lemma}

\newtheorem{corollary}{Corollary}

\theoremstyle{definition}

\newtheorem{definition}{Definition}

\newtheorem{remark}{Remark}

\newtheorem{example}{Example}

\providecommand{\keywords}[1]

\title{Linear Dynamics of composition operators on the Schwartz spaces}

\author[J. Henriquez]{Javier Henriquez-Amador}
\address{
Javier Henriquez-Amador:
\endgraf
  Department of Mathematics: Analysis, Logic and Discrete Mathematics
  \endgraf
  Ghent University, Belgium
  \endgraf
and
\endgraf
  Departamento de Matematicas
  \endgraf
  Universidad del Valle
  \endgraf
  Cali-Colombia
   \endgraf
    {\it E-mail address} {\rm javier.henriquez@ugent.be},{\rm javier.henriquez@correounivalle.edu.co}}

\author[B. Rojas]{Brian Rojas}
\address{ Brian Rojas:
\endgraf
  Departamento de Matematicas
  \endgraf
  Universidad del Valle
  \endgraf
  Cali-Colombia
   \endgraf
    {\it E-mail address} {\rm javier.henriquez@correounivalle.edu.co}}

\makeatletter

\date{\today}

\begin{document}

\begin{abstract}
In this note, we show that invertible composition operators on $S(\mathbb{R})$ are never generalized hyperbolic and, when the symbol has a fixed point, the corresponding operator fails to have the positive shadowing property. We then establish sufficient conditions on the symbol under which the associated composition operator is positively topologically expansive. In particular, we obtain a complete characterization for affine symbols and prove expansivity for broad classes of odd-degree polynomial symbols, while showing that polynomial symbols of even degree cannot generate topologically expansive operators. Our results reveal a strong connection between the dynamics of the symbol and the linear dynamics of the induced composition operator, and provide new examples and obstructions for topological expansivity in locally convex spaces.
\end{abstract} 

\maketitle
\markright{Dynamical Properties of Composition Operators on Schwartz Space}


\section{Introduction}

Among the natural classes of operators that act on $S(\mathbb{R})$, 
composition operators provide a particularly rich setting for studying these 
questions. Given a smooth symbol $\varphi:\mathbb{R}\to\mathbb{R}$, the associated 
composition operator is defined by
\[
C_\varphi f = f\circ\varphi,
\]
whenever $C_\varphi$ maps $S(\mathbb{R})$ continuously into itself if and only if the following conditions is hold:
\begin{enumerate}
        \item [C.1] For each $j\in\mathbb{N}$ there exist $C,p>0$ such that $$|\varphi^{j}(x)|\leq C \ (|1+|\varphi(x)|^{2}|)^{p}, \ \ \textrm{for all} \ \ x\in\mathbb{R}.$$ 
        
      \item[C.2] There exist $r>0$ such that $|\varphi(x)|\geq |x|^{1/r}, \ \ \textrm{for all} \ \ |x|\geq r.$    \end{enumerate}
See, for instance, \cite{galbis2018composition,asensio2025power}. 
The dynamics of $C_\varphi$ are closely related to those of the symbol $\varphi$: 
The orbits satisfy
\[
C_\varphi^n f = f\circ\varphi_n, \ \  \varphi_n=\underbrace{\varphi\circ\varphi\circ\cdots\circ\varphi}_{n\text{ times}}.
\] 
The main contribution of this work is to prove that the Schwartz space admits continuous composition operators that exhibit (positive) topological expansivity for several classes of symbols $\varphi.$ The simple more case is when $$\varphi(x)=ax+b, \ \ a\neq 0, \ \ b\in\mathbb R,$$ 
which complete characterization is possible with some consequence in the Li-Yorke chaos and topological transitivity. Within the polynomial setting, we show that no polynomial of even degree is topologically expansive by considering the orbits of compactly supported functions whose support lies outside the range of the symbol.\\

This behavior of polynomial of even degree is particularly revealing, since if we fix a symbol $\varphi:I\to I,$ for some compact interval $I$ then we can show (see Theorem \ref{133}) that a necessary and sufficient condition for $C_{\varphi}$ to fail to be positively topologically expansive is the following: for each $j$ there exist $C$ and $p$ such that

\[
|(\varphi_n)^j(x)| < C \bigl(1 + \varphi_n^2(x)\bigr)^p,
\]
for every $x \in I$ and for all $n \in \mathbb{N}.$ 
On the other hand, for polynomials of odd degree greater than one, the situation is more subtle. In this setting, we prove that if every point on the real line repels the symbol, that is $|\varphi'(x)|>1$ for all $x\in\mathbb R,$ then this condition is sufficient to generate an expansive orbit for each $0\neq f\in S(\mathbb R)$ (see Theorem \ref{132}). More generally, this condition can be relaxed when the symbol has no fixed points; In this direction, our results provide new insight into the notion of topological expansivity, recently introduced in \cite{bernardes2025generalized}, within the context of locally convex spaces.\\

On the other hand, when a symbol $\varphi$ has a fixed point, we exploit the structure of the seminorms to construct a $\delta$-pseudo-orbit that is not shadowed in $S(\mathbb{R})$ (see Theorem \ref{t1}). This reveals a notable phenomenon: a pathological behavior that arises not only in Banach spaces of analytic functions (see, for instance, \cite{Blois2025,blois2026shadowing,Alvarez2025,alvarez2025spectra}), but also extends to certain locally convex spaces. Regarding the generalized hyperbolicity property, we show that a decomposition of the Schwartz space for an invertible composition operator is impossible since it would necessarily reduce to a trivial decomposition (see Theorem \ref{hyperb} ). \\

Recent work has investigated several dynamical properties of composition 
operators in the Schwartz space. In particular, the dynamics and spectra of 
composition operators in $S(\mathbb{R})$ have been studied in 
\cite{FERNANDEZ20183503,FERNANDEZ2020107052}, where the authors showed that the Schwartz spaces do not support composition operators supercyclic or hypercyclic.In addition, they
provided a complete characterization of power-boundedness and mean ergodicity
in the same setting. The Schwartz space also supports other types of operators with interesting dynamical properties; see, for example, \cite{albanese2022spectra}. On the other hand, in the context of locally convex spaces, the dynamics of composition operators have recently been studied in several settings. For a more detailed account, see \cite{albanese2022dynamics,kalmes2022mean,FERNANDEZ20183503,galbis2018composition,albanese2023spectral,FERNANDEZ2020107052}.\\

This work is organized as follows: In section \ref{sec2} we introduced some notion of linear dynamics in Fréchet space. In section \ref{sec3} we studied the Shadowing and Generalized Hyperbolicity in the Schwartz space $S(\mathbb{R})$. Finally, in section \ref{sec4} we showed that the Schwartz space $S(\mathbb{R})$ supports composition operators with positively expansive.  

\section{Preliminaries
}\label{sec2}

Throughout this section, $\mathbb K$ denotes $\mathbb R$ or $\mathbb C$.
Let $X$ be a Hausdorff locally convex space over $K$, and let
$T\in L(X)$ be a continuous linear operator on $X$. We denote by
$GL(X)$ the set of all continuous linear operators on $X$ having a
continuous inverse. We assume that the topology of $X$ is induced by a
directed family of seminorms
\[
    \left(\|\cdot\|_{\alpha}\right)_{\alpha\in I}.
\]
Recall that a family of seminorms is directed if, for every
$\alpha,\beta\in I$, there exists $\gamma\in I$ such that
\[
    \|\cdot\|_{\alpha}\leq \|\cdot\|_{\gamma}
    \quad\text{and}\quad
    \|\cdot\|_{\beta}\leq \|\cdot\|_{\gamma}.
\]
 Let $p_n:=\|\cdot\|_{n},$ the metric
\[
d(x,y)
:=
\sum_{n=1}^{\infty}
2^{-n}\frac{p_n(x-y)}{1+p_n(x-y)},
 \ \ x,y\in X,
\]
is well defined and induces the original topology of $X.$ In particular, the topology generated by $d(\cdot,\cdot)$ coincides with the locally convex topology generated by the family $((p_n)_{n\in\mathbb N}).$ When $X=S(\mathbb R)$ we have the seminorms family $$\|f\|_n:=\max_{1\leq j \leq n} \  \sup_{x} \  (1+|x|^2)^n \ |f^{(j)}(x)|, \ \ f\in S(\mathbb R).$$

The following definitions will be used throughout this paper, see \cite{bernardes2025generalized} for more details.

\begin{definition}[Generalized hyperbolicity]
An operator $T\in L(X)$ is said to be \emph{generalized hyperbolic} if there
exists a topological direct sum decomposition
\[
    X=M\oplus N
\]
such that the following conditions hold:

\begin{enumerate}[label=\textnormal{(GH\arabic*)}]
    \item $T(M)\subseteq M$;

    \item $T(N)\supseteq N$ and the restriction
    \[
        T|_N:N\longrightarrow T(N)
    \]
    is an isomorphism;

    \item for every $\alpha\in I$, there exist $\beta\in I$, $c>0$,
    and $t\in(0,1)$ such that
    \[
        \|T^n y\|_{\alpha}
        \leq c\,t^n\|y\|_{\beta},
        \qquad y\in M,\quad n\in \mathbb N_0,
    \]
    and
    \[
        \|S^n z\|_{\alpha}
        \leq c\,t^n\|z\|_{\beta},
        \qquad z\in N,\quad n\in \mathbb N_0,
    \]
    where
    \[
        S=(T|_N)^{-1}:T(N)\longrightarrow N
    \]
    is restricted to $N$.
\end{enumerate}
\end{definition}

If, in addition, $T$ is invertible, condition
\textnormal{(GH2)} is equivalent to
\[
    T^{-1}(N)\subseteq N,
\]
and the second inequality in \textnormal{(GH3)} can be written as
\[
    \|T^{-n}z\|_{\alpha}
    \leq c\,t^n\|z\|_{\beta},
    \qquad z\in N,\quad n\in \mathbb N_0.
\]
If both $M$ and $N$ are $T$-invariant, then $T$ is called
\emph{hyperbolic}. Thus, generalized hyperbolicity extends the usual
hyperbolic splitting to the locally convex setting.\\

Let $U$ be a neighborhood of $0$ in $X$. A finite or infinite sequence
$(x_j)$ is called a \emph{$U$-pseudotrajectory} of $T$ if
\[
    Tx_j-x_{j+1}\in U
\]
for every pair of consecutive indices for which the expression is
defined. A finite $U$-pseudotrajectory
\[
    (x_j)_{j=0}^{k}
\]
is also called a \emph{$U$-chain}. If, in addition, $x_k=x_0$, it is
called a \emph{$U$-cycle}.

\begin{definition}[Shadowing property]
An operator $T\in GL(X)$ is said to have the \emph{shadowing property} if,
for every neighborhood $V$ of $0$ in $X$, there exists a neighborhood
$U$ of $0$ in $X$ such that every $U$-pseudotrajectory
$(x_j)_{j\in\mathbb Z}$ is $V$-shadowed by the trajectory of some
$x\in X$, that is,
\[
    x_j-T^j x\in V,
    \qquad j\in\mathbb Z.
\]
When $T\in L(X)$ is not invertible, one usually considers the corresponding
\emph{positive shadowing property}: every $U$-pseudotrajectory
$(x_j)_{j\in\mathbb N_0}$ is $V$-shadowed by some $x\in X$, namely
\[
    x_j-T^j x\in V,
    \qquad j\in\mathbb N_0.
\]
\end{definition}

For later reference, the finite version is defined analogously: $T$
has the \emph{finite shadowing property} if every sufficiently accurate
finite chain is shadowed, with the required accuracy depending only on
the prescribed neighborhood of $0$. The relationship between hyperbolicity and shadowing follows result (\cite[Theorem 6]{bernardes2025generalized})  

\begin{theorem}\label{prelim}
  Suppose that the topology of a locally convex space X is induced by a directed family $(\|\cdot\|_{\alpha})_{\alpha\in I}$ of complete seminorms. If $T\in L(X)$  is generalized hyperbolic  and the set $Ker(\|\cdot\|_{\alpha}):=\{f\in X: \|f\|_{\alpha}=0\}$ is $T$-invariant for all $\alpha\in I$, then $T$ has the positive shadowing property.
\end{theorem}

For operators on normed spaces, expansivity is usually expressed in
terms of the unboundedness of the orbit of every nonzero vector. In
general, locally convex spaces, the seminorm-based formulation is more
appropriate (see \cite[Definition 32]{bernardes2025generalized}).

\begin{definition}[Topological expansivity]
Let $T\in GL(X)$. We say that $T$ is \emph{topologically expansive} if,
for every nonzero $x\in X$, there exists $\alpha\in I$ such that
\[
    \sup_{n\in\mathbb Z}\|T^n x\|_{\alpha}=\infty.
\]
Equivalently, the orbit
\[
    \operatorname{Orb}(x,T)
    :=\{T^n x:n\in\mathbb Z\}
\]
is topologically unbounded for every $x\neq0$. When $\mathbb Z$ is replaced by $\mathbb N$ we say that $T$ is positively topologically expansive.
\end{definition}

\begin{remark}
If $X$ is a normed space, the preceding condition reduces to
\[
    \sup_{n\in\mathbb Z}\|T^n x\|=\infty
    \qquad\text{for every }x\neq0.
\]
For metrizable locally convex spaces, topological expansivity is related
to, but need not be equivalent to, the usual metric notion of
expansivity.
\end{remark}

\begin{definition}[Li--Yorke pair]
Let $T\in L(X)$. Two distinct points $x,y\in X$ form a
\emph{Li--Yorke pair} if the following conditions hold:

\begin{enumerate}[label=\textnormal{(LY\arabic*)}]
    \item for every neighborhood $V$ of $0$ in $X$, there exists
    $n\in\mathbb N$ such that
    \[
        T^n x-T^n y\in V;
    \]

    \item there exists a neighborhood $U$ of $0$ in $X$ such that
    \[
        T^n x-T^n y\notin U
    \]
    for infinitely many $n\in\mathbb N$.
\end{enumerate}
\end{definition}

\begin{definition}[Li--Yorke chaos]
The operator $T$ is said to be \emph{Li--Yorke chaotic} if there exists
an uncountable set $S\subseteq X$ such that every pair of distinct
points in $S$ is a Li--Yorke pair.
\end{definition}

When $X$ is a metric space with metric $d$, the preceding definition
takes the familiar form
\[
    \liminf_{n\to\infty}d(T^n x,T^n y)=0,
    \qquad
    \limsup_{n\to\infty}d(T^n x,T^n y)>0.
\]
In the setting of continuous linear operators on Fr\'echet spaces,
Li--Yorke chaos is also characterized by the existence of a
semi-irregular vector; that is, a vector whose orbit does not converge
to zero but has a subsequence converging to zero.

\begin{definition}[Topological transitivity]
An operator $T\in L(X)$ is said to be \emph{topologically transitive} if,
for every pair of nonempty open subsets $A,B\subseteq X$, there exists
$n\in\mathbb N_0$ such that
\[
    T^n(A)\cap B\neq\varnothing.
\]
\end{definition}

More generally, $T$ is called \emph{topologically mixing} if, for every
pair of nonempty open subsets $A,B\subseteq X$, there exists
$n_0\in\mathbb{N}_0$ such that
\[
    T^n(A)\cap B\neq\varnothing
    \qquad\text{for every }n\geq n_0.
\]
Thus, topological mixing implies topological transitivity. One relationship between Expansivity and Li-Yorke chaos is the following (see \cite[Theorem 38]{bernardes2025generalized}).

\begin{theorem}\label{prely2} A uniformly topologically expansive operator on a locally convex space is neither Li-Yorke chaotic nor topologically transitive.
\end{theorem}

\section{Shadowing and Generalized hyperbolic for \texorpdfstring{$C_{\varphi}$}{C varphi}}\label{sec3}

In this section, we analyze the notion of shadowing and its implications for generalized hyperbolicity. 
\begin{theorem}\label{t1}
Let $\varphi$ be a symbol. If there exists $x_0\in\mathbb{R}$ such that $\varphi(x_0) = x_0$ then  $C_{\varphi}$ does not have the positive shadowing property on $S(\mathbb{R}).$
\end{theorem}

\begin{proof}
The analysis depends essentially of value $\varphi'(x_0).$ First, let $|\varphi'(x_0)| > 1;$ Take $f \in S(\mathbb{R})$ such that $f'(x_0)\neq 0,$ $0<d(f,0)<\delta$ and fixed $\epsilon_0 \in (0,\frac{1}{|\varphi'(x_0)-1|}).$ For each $\delta > 0,$ consider the $\delta$-pseudo-trajectory $f_{n,\delta}$ given by
\begin{equation}\label{seq1}
    f_{n,\delta}(x):=\displaystyle\frac{\delta}{d(f,0)}\sum_{j=0}^{n} (f\circ\varphi_{j})(x), \ \ x\in\mathbb{R}.
\end{equation}
 Hence, for any $h \in S(\mathbb{R})$ we have 

\begin{equation*}
\begin{array}{ll}
\|C^n_{\varphi}h-f_{n,\delta}\|_{\pi_k} &\geq |(C^n_{\varphi}h)'(x_0)-f'_{n,\delta}(x_0)| \\
& \geq \left|h'(x_0)(\varphi'(x_0))^n-\frac{\delta}{d(f,0)}f'(x_0)\displaystyle\frac{(\varphi'(x_0))^n-1}{(\varphi'(x_0))-1}\right|\\
& \geq \left||(\varphi'(x_0))^n(h'(x_0)-\displaystyle\frac{\delta f'(x_0)}{d(f,0)(\varphi'(x_0)-1)})|-|\frac{\delta}{d(f,0)(\varphi'(x_0)-1)}| \right|
\end{array}
\end{equation*}
In the case $$h'(x_0)-\frac{\delta f'(x_0)}{d(f,0)(\varphi'(x_0)-1)}=0,$$
it follows $\|C^n_{\varphi}h-f_{n,\delta}\|_k\geq \epsilon_0$ for any $n\in\mathbb{N} $ and any $k\in\mathbb{N}.$ In the other case, note that $|(\varphi'(x))^{n}|\to+\infty $ as $n\to+\infty$ implies $$\displaystyle\liminf_{n\to+\infty}\|C^n_{\varphi}h-f_{n,\delta}\|_k=+\infty, \ \  \textrm{for each}  \ \ k\in\mathbb{N}.$$
In any case, 
\begin{eqnarray*}
d(C^{n}_{\varphi}h,f_{n,\delta})&=& \sum_{k=0}^{+\infty} \frac{1}{2^k} \ \frac{\|(C^{n}_{\varphi}h-f_{n,\delta}\|_k}{1+\|(C^{n}_{\varphi}h-f_{n,\delta}\|_k} \\
   &\geq& \frac{\epsilon_0}{1+\epsilon_0} \sum_{k=0}^{+\infty} \frac{1}{2^{k}} \\
   &=& \frac{2\epsilon_0}{1+\epsilon_0}, 
\end{eqnarray*}
 for all $n\in\mathbb{N}$ thus $C_\varphi$  does not have the positive shadowing property on $S(\mathbb{R}).$ For the case $ |\varphi'(x_0)| < 1$ note that since $|(\varphi'(x_0))^{n}|\to0$ as $n\to+\infty$ the analysis is exactly same. Finally, let \(\varphi'(x_0) = 1\); the following inequality is obtained
$$\|C^n_{\varphi}h-f_{n,\delta}\|_k \geq |(C^n_{\varphi}h)'(x_0)-f'_{n,\delta}(x_0)| \\
 \geq \left|h'(x_0)-\frac{\delta}{\delta(f,0)}f'(x_0)n\right|, \ \ n\in\mathbb{N}.$$
Since $f'(x_0)\neq 0$ for each $k\in\mathbb{N}$ we get $$\liminf_{n\to+\infty}\|C^n_{\varphi}h-f_{n,\delta}\|_k=+\infty.$$ This shows that $C_{\varphi}$ also does not have the positive shadowing property on $S(\mathbb R).$ 
\end{proof}

\begin{remark}\label{r1*} It is easy to show that $C_{\varphi}$ has a $C_{\varphi}$-invariant kernel for all indices. Hence, if the symbol $\varphi$ has fixed point $x_0$ then by Theorem \ref{prelim} the operator $C_\varphi$ does not generalized hyperbolic on $S(\mathbb R).$
\end{remark}

\begin{theorem}\label{t2}
Let $\varphi:\mathbb{R}\to\mathbb{R}$ be a polynomial of even degree. Then $C_\varphi$ does not have the positive shadowing property on $S(\mathbb R).$
\end{theorem}

\begin{proof} If the symbol $\varphi$ has a fixed point then  $C_{\varphi}$  does not possess the positive shadowing property, by Theorem \ref{t1}. Hence, it is sufficient to analyze even-degree polynomial symbols without fixed points. Indeed, first we construct a $\delta$-pseudotrajectory that will be employed throughout the proof.
Let $0\neq g\in S(\mathbb{R})$ and suppose that
\[
C_{\varphi}g-g=f,
\]
for some $0\neq f\in S(\mathbb{R})$. By induction,
\[
g=C_{\varphi}^{\,n}g-\sum_{k=0}^{n-1}C_{\varphi}^{\,k}f,
\qquad n\geq 1.
\]
Multiplying both sides by $\delta_0/d(f,0)$ for $\delta_0>0$, we have
\[
\frac{\delta_0}{d(f,0)}\,g
=
\frac{\delta_0}{d(f,0)}\,C_{\varphi}^{\,n}g-f_n, \ \ n\geq 1
\]
where
\[
f_n:=
\frac{\delta_0}{d(f,0)}
\sum_{k=0}^{n-1}C_{\varphi}^{\,k}f, \ \ n\geq 1.
\]
Assume that $C_{\varphi}$ possesses the positive shadowing property. Let $\varepsilon\in(0,1);$ take $\delta>0$ such that for every $\delta$-pseudotrajectory $(f_n)_{n\geq 0}$, there exists $0\neq h\in S(\mathbb{R})$ satisfying $d(C_{\varphi}^n h,f_n)<\varepsilon, \ \ n\in\mathbb N$ which implies that 
\[
\|C_{\varphi}^{\,n}h-f_n\|_{1}<\varepsilon,
\qquad n\geq 0.
\]

Applying the shadowing property to the $\delta$-pseudotrajectory with $\delta=\delta_0$ constructed above, we obtain
\[
\left\|
C_{\varphi}^{\,n}h-
\left(
\frac{\delta_0}{d(f,0)}
(C_{\varphi}^{\,n}g-g)
\right)
\right\|_{1}
<\varepsilon,
\qquad n\geq 0.
\]

Since $\varphi$ is a polynomial of even degree we have  $C_{\varphi_n}h, C_{\varphi_n}g\to 0$ in $S(\mathbb{R})$ (see
\cite{FERNANDEZ20183503}). Hence, 
\[
\left\|
\frac{\delta_0}{d(f,0)}\,g
\right\|_{1}
\leq \varepsilon.
\]
Since $\|g\|_{1}\neq 0$, it follows that
\[
\|g\|_{1}
\leq \frac{d(C_{\varphi}g,g)}{\delta_0}\varepsilon\le \frac{\varepsilon}{\delta_0},
\]

Since this estimate must hold for every
$g\in S(\mathbb{R})$, letting
$\|g\|_{1}\to\infty$ leads to a contradiction.
Therefore, $C_{\varphi}$ cannot possess the shadowing property.
\end{proof}

\begin{theorem}\label{hyperb}
Let $\varphi$ be a surjective symbol. Then the composition operator $C_{\varphi}$ is not generalized hyperbolic on $S(\mathbb R).$
\end{theorem}

\begin{proof}
Without loss of generality, assume that $\varphi$ has no fixed points. Suppose, by contradiction, that $\varphi$ is generalized hyperbolic. Then there exist closed subspaces $N,M\subseteq S(\mathbb{R})$ such that
\[
S(\mathbb{R})=N\oplus M.
\]
Let $f\in M.$ Then
\[
\|C_\varphi^n f\|_{1}\longrightarrow 0, \ \ \textrm{as} \ \ n\to +\infty.
\]

Since $\varphi$ is surjective, for every $y\in \operatorname{supp}(f)$ there exists an unbounded monotone sequence
$\{x_n\}_{n\in\mathbb N}$ such that
$\varphi_n(x_n)=y.$
Hence,
\[
|f(\varphi_n(x))-f(y)|
=
\left|
\int_{x_n}^{x}
\frac{d}{dr}(f\circ \varphi^n)(r)\,dr
\right|
\leq
\int_{-\infty}^{\infty}
\left|
\frac{d}{dr}(f\circ \varphi^n)(r)
\right|\,dr.
\]

Since $\varphi$ has no fixed point, for each $x\in\mathbb R$ we get $f(\varphi_n(x))\to 0$ as $n\to+\infty$ and the dominated convergence theorem yields
$f(y)=0.$ As $y\in \operatorname{supp}(f)$ it is a contradiction, unless $f=0$ but this implies $M=\{0\}.$  Therefore, $C_\varphi$ is invertible on $S(\mathbb{R})$, that is 
\[
C_\varphi^{-1}:S(\mathbb{R})
\longrightarrow
S(\mathbb{R}),
\]
Hence, for every $g\in C^{\infty}_c(\mathbb{R})$,
\[
C_\varphi C_\varphi^{-1}g
=
C_\varphi^{-1}C_\varphi g.
\]
 We show that $\varphi$ is injective. Let
$x_1,x_2\in\mathbb{R}$ satisfy
$
\varphi(x_1)=\varphi(x_2);
$ for every $g\in C_c(\mathbb{R})$ we have 
\[
(C_\varphi^{-1}g)(\varphi(x_1))
=
g(x_1)
=
g(x_2)
=
(C_\varphi^{-1}g)(\varphi(x_2)).
\]
Choose $h\in C^{\infty}_c(\mathbb{R})$ such that
$h=1$ on $[x_1,x_2]$, and define
$g(x)=x\,h(x).
$
Substituting this function into the previous identity yields
$x_1=x_2.
$
Hence $\varphi$ is bijective and this implies 
$C_\varphi^{-1}
=
C_{\varphi^{-1}}$ and $\varphi^{-1}$ is also a symbol. Finally, We now observe that $C_{\varphi^{-1}}$ is hyperbolic with the decomposition
\[
M'=S(\mathbb{R}),
\qquad
N'=\{0\}.
\]
However, applying the same analysis above to the symbol $\varphi^{-1}$ yields
\[
M'=\{0\},
\]
which is impossible. 
Therefore, $\varphi$ is not generalized hyperbolic.
\end{proof}

\begin{corollary}
If $\varphi$ is a polynomial symbol, the composition operator $C_{\varphi}$ does not possess the positive shadowing property on $S(\mathbb{R}).$
\end{corollary}

\section{Topological expansiveness and Li-Yorke chaos of \texorpdfstring{$C_{\varphi}$}{C varphi}}\label{sec4}

In this section, we provide a first approach to the study of topological expansivity for a class of symbols that included polynomial symbols. 

\begin{theorem}\label{th5}
    Let $\varphi(x) = ax + b$ be a symbol, where $a,b\in \mathbb{R}$ with $a\neq 0$. Then, $C_{\varphi}$ is topologically expansive in $S(\mathbb{R})$ when $a\neq1$ or $a=1$ and $b\neq0.$  In the case $a=-1$ and $b\neq0$ $C_\varphi$ or $a=1$ and $b=0$ does not have a topological expansive in $S(\mathbb{R}).$
\end{theorem}

\begin{proof} Let $a\neq \pm 1,$ for $x\in\mathbb{R}$ we have
$$\varphi_n(x):=\left\{\begin{array}{cc}
  a^n x+\displaystyle\frac{a^n-1}{a-1}b,   &  b\neq 0 \\
a^{n}x, & b=0
\end{array}
\right.
$$
and
$$\varphi_{-n}(x):=\left\{\begin{array}{cc}
  a^{-n} x+\displaystyle\frac{a^{-n}-1}{a^{-1}-1} a^{-1}b,   &  b\neq 0 \\
a^{-n}x,     & b=0
\end{array},
\right.$$
Let us analyze the positive orbits; the analysis for the negative orbits is analogous.
 Let $f \in S(\mathbb{R})$ such that $f'(x_0) \neq 0$ for some $x_0\in\mathbb{R}.$ In the case, $b\neq 0$ we can define $$z_{n}:=a^{-n} \left(x_0 - \frac{a^{n} - 1}{a - 1}b \right), \ \ \ n=1,2,3,\cdots$$
then, for $m\in\mathbb{N}$ we get 
$$|a|^n \ (1+|z_{n}|^2) \ |f'(x_0)| \leq |a|^n \ (1+|z_{n}|^2)^{m} \ |f'(\varphi_{n}(z_{n}))|\leq  \|C_{\varphi}^n f\|_m.$$
 Since  $$|a|^n \ (1+|z_{n}|^2)\to +\infty, \ \ \textrm{as} \ \ n\to+\infty,$$
 we obtained $$\|C_{\varphi}^n f\|_m\to +\infty, \ \ \textrm{as} \ \ n\to+\infty$$ that is, $C_\varphi$ is expansive topologically on $S(\mathbb{
 R}
 ).$ On the other hand, If $a =1$ and $b\neq0$ then $\varphi_n(x)=x+nb,$ thus 
$$(1 + (x_0 +n b)^2)^m |f'(x_0)| \leq \|C_{\varphi}^n f\|_m.$$ Hence, $C_{\varphi}$ is expansive topologically on $S(\mathbb{
 R}
 ).$ The cases $b=0$ are similar. Finally, when $a=-1$ and $b\neq 0$ we have $\varphi_{2n}(x)=x$ this implies $\|C_{\varphi}^{2n}f\|_m=\|f\|_m$ thus $\sup_n\|C_{\varphi}^{2n}f\|_m<+\infty,$ that is, $C_{\varphi}$ is not expansive topologically on $S(\mathbb{
 R}
 ).$ Similarly, when $a=1$ and $b=0.$
\end{proof}

\begin{remark} Some consequence of expansiveness property for the simple polynomial case:
    \begin{enumerate}
        \item[$\bullet$]
    Theorem \ref{th5} provided a counterexample of \cite[Corollary 36]{bernardes2025generalized}. Moreover, by \cite[Theorem 38] {bernardes2025generalized} we have that $C_\varphi$ neither Li-Yorke chaotic nor topologically transitive provided $a\neq 1.$
    \item[$\bullet$] By Theorem~\ref{th5}, any orbit of $0\neq f\in S(\mathbb R)$ to expanded. Hence, argue as in Theorem~\ref{t2} we can deduced that for any polynomial symbol the operator $C_{\varphi}$ does not has the positive Shadowing property on $S(\mathbb R).$
    \item[ $\bullet$] The question for polynomial $\varphi$ of degree major that one is a natural one. However, the even degree case is easy since is enough to take $f\in C_c^{\infty}(\mathbb R)$ such that $\textrm{supp}(f)\cap \varphi(\mathbb R)=\emptyset$ then $C_\varphi^n f=0$ because $\textrm{supp}(f)\cap \varphi_n(\mathbb R)\subset \textrm{supp}(f)\cap \varphi(\mathbb R)$ for all $n.$  
\end{enumerate}
\end{remark}

In the following theorem we show a sufficient condition for the generated expansive orbit that included the odd degree polynomial case. 

\begin{theorem}\label{132}
Let $\varphi:\mathbb{R}\to\mathbb{R}$ be a symbol, and denote by $X$ the set of fixed point for $\varphi.$
\begin{enumerate}
    \item[(1)] If $C_\varphi\in  GL(S(\mathbb R))$ and $X$ is not empty then $C_\varphi$ is positive expansive topologically on $S(\mathbb R)$ provided that \[\min_{x\in  \mathbb R}|\varphi'(x)|>1.\]
    
    \item[(2)]  Assume $X=\emptyset,$ $\varphi$  surjective map and suppose that there exists $k>0$ such that
\[
|\varphi'(x)|\ge 1,\qquad x\notin[-k,k],
\]
and that $\varphi'$ has only countably many zeros. Then the composition operator $C_\varphi$ is topologically expansive.
\end{enumerate}
\end{theorem}

\begin{proof}
$(1).$ Suppose that there exists $0\neq f\in S(\mathbb R)$ such that
$$M:=\sup_{n}\|C^n_{\varphi}f'\|_{1}\leq \sup_{n}\|C_\varphi^n f\|_{1}<+\infty.$$
Since $\varphi$ is invertible, for $y\in \textrm{supp}(f)$ with $f'(y)\neq0$ we can find a monotonic sequence $\{x_n\}$ such that $\varphi_n(x_n)=y$ and $\varphi(x_{n+1})=x_n.$ In the case where $x_n\to x_0$ we get $x_0\in X$ thus
\[|f'(x_0)-f'(y)|=|f'(\varphi_n(x_0))-f'(\varphi_n(x_n))|\leq \int_{[x_0,x_n]} \left|\frac{d}{dr}(f'\circ\varphi_n)(r)\right| \ dr\leq \ M'_1 \ |x_n-x_0|.\]
Taking $n\to+\infty$ we get $f'(y)=f'(x_0)$ which implies $f'(x_0)\neq 0.$ From $\varphi'(x)>1$ we deduced that $\varphi'_n(x)>1$ for all $n$ and for all $x$. Hence,

\[(1+|x_0|^2) \ |f'(x_0)| \ (\varphi'(x_0))^{n-1} \leq M', \ \ n=1,2,\cdots.\]
It is a contradiction, because $(\varphi'(x_0))^{n-1}\to +\infty,$ as $n\to+\infty$ . Therefore, the sequence $\{x_n\}$ satisfies
$$|x_n|\to+\infty, \ \ \textrm{as} \ \ n\to+\infty.$$ Since $|\varphi'_n(x_{n+1})|>1$ we have \[(1+|x_n|^2) \ |f'(y)|\leq (1+|x_n|^2) \ |f'(y)| \ |\varphi'_n(x_{n+1})|\leq \sup_{n}\|C_\varphi^n f\|_{\pi_1}.\] 
That is,
\[\sup_{n}\|C_\varphi^n f\|_{1}=+\infty,\] a contradiction. Consequently, $C_\varphi$ is a positive topological expansive in $S(\mathbb R).$ To prove 
  $(2);$  Without loss of generality, assume that $\varphi$ has no fixed points. Since $\varphi$ is surjective, we may suppose that
\[
x<\varphi(x), \qquad x\in\mathbb{R}.
\]

Let $f\in S(\mathbb{R})$ be nonzero and choose
$y\in \operatorname{supp}(f)$ such that $f'(y)\neq 0$.
Since $\varphi$ is surjective, there exists a decreasing sequence
$\{x_n\}_{n\in\mathbb N}$ satisfying
\[
\varphi^n(x_n)=y,
\qquad
\varphi(x_{n+1})=x_n,
\]
and therefore $x_n\to -\infty$.

Using the definition of the seminorms of $S(\mathbb{R})$, we obtain
\[
(1+x_n^2)\,
\bigl|f'(y)\varphi_n'(x_n)\bigr|
\leq
\|C_{\varphi^n}f\|_{1}.
\]

Moreover,
\[
\varphi_n'(x_n)
=
\prod_{j=1}^{n}\varphi'(x_j).
\]

Since $x_n\to -\infty$ and $|\varphi'(x)|\ge1$ for
$x\notin[-k,k]$, there exists $N\in\mathbb N$ such that
$x_n\notin[-k,k]$ for all $n\geq N$. Hence
\[
|\varphi_n'(x_n)|
=
\prod_{j=1}^{n}|\varphi'(x_j)|
\longrightarrow\infty \text{ or  }|\varphi'_n(x_n)|\ge |\varphi'_N(x_N)|.
\]

Consequently,
\[
\|C_{\varphi^n}f\|_{1}
\longrightarrow\infty.
\]

It remains to show that such a point $y$ exists. Suppose, by contradiction,
that for every $y\in\operatorname{supp}(f)$ and every backward orbit
$\{x_n\}_{n\in\mathbb N}$ satisfying $\varphi^n(x_n)=y$, there exists
$n\in\mathbb N$ such that
\[
\varphi'(x_n)=0.
\]

Since $\varphi'$ has only countable many zeros, say
$z_1,\ldots,z_m,\ldots$, it follows that
\[
\operatorname{supp}(f)
\subset
\bigcup_{j=1}^{\infty} O(z_j),
\]
where $O(z_j)$ denotes the orbit of $z_j$ under $\varphi$. Since each orbit is countable, the set
\[
\bigcup_{j=1}^{\infty} O(z_j)
\]
is countable. On the other hand, $\operatorname{supp}(f)$ is uncountable
because $f\neq 0$ and $f$ is continuous. This contradiction proves that
there exists $y\in\operatorname{supp}(f)$ and a backward orbit
$\{x_n\}_{n\in\mathbb N}$ such that
\[
\varphi'(x_n)\neq 0,
\qquad n\in\mathbb N.
\]

Therefore, $C_\varphi$ is topologically expansive.

\end{proof}

\begin{example}
    Consider 
\[
\varphi(x):=
\left\{
\begin{array}{ll}
x+e^{-\frac{1}{1-x^2}}, & x\in [-1,1],\\[6pt]
x, & x\notin [-1,1].
\end{array}
\right.
\]
    Observe that $\varphi\in C^{\infty}(\mathbb{R})$. Moreover, if $x\notin [-1,1]$, then the following holds:
    \[
    |x|\le |\varphi(x)|.
    \]
    It is not difficult to verify that all the derivatives of $\varphi$ are bounded. Therefore, $\varphi$ is a symbol of the composition operator. Since it satisfies the hypotheses of the previous theorem, the corresponding composition operator is topologically expansive on $S(\mathbb R).$
\end{example}

\begin{lemma}
\label{exf}
    Let $I\subset\mathbb{R}^n$ a bounded interval and $j\in \mathbb{N}$ then there exists $f\in S(\mathbb{R})$ such that
\[
f^j(x)>1, \ \text{  for  } x\in I.
\]
\end{lemma}

\begin{proof}
Consider the function $f(x)=e^{-x^2}$. Then, for every integer $n \geq 0$, its $n$-th derivative satisfies
\[
f^{(n)}(x)=P_n(x)e^{-x^2},
\]
where $P_n$ is a polynomial of degree $n$. Hence, the set of zeros of $P_n$, denoted by $K_n$, is finite. Consequently, there exists an open interval $I \subset \mathbb{R}$ such that
\[
P_n(y)>0 \quad \text{for all } y \in I.
\]

Therefore, one can choose $c \in \mathbb{R}$ such that
\[
f^{(n)}(x-c) > f_0 \quad \text{for all } x \in I,
\]
for some constant $f_0>0$.

Now define
\[
g(x)=\frac{2}{f_0}f(x-c).
\]
On the one hand, $g \in S(\mathbb{R})$. Moreover,
\[
g^{(n)}(x)=\frac{2}{f_0}f^{(n)}(x-c),
\]
and therefore
\[
g^{(n)}(x)>1, \quad \text{for all } x \in I.
\]

This completes the proof.
\end{proof}

\begin{theorem}
\label{133}
Let $\varphi$ be a symbol such that there exists a compact interval $I$ satisfying $\varphi:I\to I$ then
for each $j \in \mathbb{N}$, there exist constants $C > 0$ and $p > 0$ such that
\[
|(\varphi_n)^j(x)| < C \bigl(1 + \varphi_n^2(x)\bigr)^p,
\]
for every $x \in I$ and for all $n \in \mathbb{N}$ if and only if $C_{\varphi}$ is not expansive on $S(\mathbb R).$
\end{theorem}
\begin{proof}
     Assume that the symbol is not topologically expansive. Then, for each 
$l \in \mathbb{N}$ and every $f \in S(\mathbb{R})$, there exists 
$C_l > 0$ such that
\[
\|C_{\varphi_n} f\|_l < C_l,
\]
for all $n \in \mathbb{N}$. Let $f$ be such that $I \subset \operatorname{supp}(f)$. Then, for each 
$l$, there exists $C_l > 0$ such that $I \subset \operatorname{supp}(f)$ 
and the derivatives of $f$ do not vanish on $I$. Hence, for example of \textbf{Lemma} \ref{exf},
\[
(1+x^2)\, |(C_{\varphi_n}f)'(x)| < C_l,
\]
for $x \in I$, by Faà di Bruno's formula.
        $|\varphi_n^j(x)|<C$
        For  $j=1,\cdots ,l$ y $n\in \mathbb{N}$.
        Reciprocally, 
consider $f \in S(\mathbb{R})$ such that $f \neq 0$ and 
$\operatorname{supp}(f) \subset I$. Since 
\[
\operatorname{supp}(f \circ \varphi_n) \subset \varphi(I),
\qquad n \in \mathbb{N},
\]
we observe that
\begin{equation}
\begin{array}{cc}
(1+x^2)^m |(f \circ \varphi_n)^{(j)}(x)|
& \leq
(1+x^2)^m
\displaystyle\sum_{k}
\binom{j}{k}
f^{(k)}(\varphi_n(x))
\frac{\varphi_n^{(k_1)}(x)}{1!}
\cdots
\frac{\varphi_n^{(k_j)}(x)}{j!}.
\end{array}
\end{equation}

It follows that for $m \in \mathbb{N}$,
\[
\|C_{\varphi_n} f\|_m
\leq
C_{m,I}\,\|f\|_{r_m},
\]
for all $n \in \mathbb{N}$.
\end{proof}

\begin{example}
    Let $\varphi(x) = x^k$, where $k$ is an odd integer greater than $2$. 
Then $\varphi$ is topologically expansive on $S(\mathbb R).$
\end{example}

\begin{proof}
Observe that $\varphi_n(x) = x^{k^n}$. Let $f \in \mathcal{S}(\mathbb{R})$, 
and suppose that $y \in \operatorname{supp}(f)$ with $y \neq 0$. 
Then there exists $x_n$ such that
\[
\varphi_n(x_n) = y.
\]
In fact,
\[
x_n = y^{\frac{1}{k^n}}.
\]
Now observe that
\begin{align*}
    (1+x_n^2)|f'(\varphi_n(x_n))(\varphi_n'(x_n))|&= (1+x_n^2)|f'(y)(\varphi_n'(x_n))|\\
    &=  (1+y^{\frac{2}{k^n}})|f'(y)k^ny^{\frac{k^n-1}{k^n}}|\\
    &\le \|C_{\varphi_n}f\|_{1}
\end{align*}
Consequently, $\varphi$ is topologically expansive.
\end{proof}

\section*{Acknowledgments}
The authors thank to the anonymous referees for their valuable comments to improve this article.

\bibliography{Bibliography}

@article{FERNANDEZ20183503,
title = {Dynamics and spectra of composition operators on the Schwartz space},
journal = {Journal of Functional Analysis},
volume = {274},
number = {12},
pages = {3503-3530},
year = {2018},
issn = {0022-1236},
doi = {https://doi.org/10.1016/j.jfa.2017.11.005},
url = {https://www.sciencedirect.com/science/article/pii/S0022123617304342},
author = {Carmen Fernández and Antonio Galbis and Enrique Jordá}
}

@article{galbis2018composition,
  title={Composition operators on the Schwartz space},
  author={Galbis, Antonio and Jord{\'a}, Enrique},
  journal={Revista matem{\'a}tica iberoamericana},
  volume={34},
  number={1},
  pages={397--412},
  year={2018}
}

@article{FERNANDEZ2020107052,
title = {Spectrum of composition operators on S(R) with polynomial symbols},
journal = {Advances in Mathematics},
volume = {365},
pages = {107052},
year = {2020},
issn = {0001-8708},
doi = {https://doi.org/10.1016/j.aim.2020.107052},
url = {https://www.sciencedirect.com/science/article/pii/S0001870820300773},
author = {Carmen Fernández and Antonio Galbis and Enrique Jordá}
}

@article{bernardes2025generalized,
  title={Generalized hyperbolicity, stability and expansivity for operators on locally convex spaces},
  author={Bernardes Jr, Nilson C and Caraballo, Blas M and Darji, Udayan B and F{\'a}varo, Vin{\'\i}cius V and Peris, Alfred},
  journal={Journal of Functional Analysis},
  volume={288},
  number={2},
  pages={110696},
  year={2025},
  publisher={Elsevier}
}

@article{asensio2025power,
  title={Power boundedness and related properties for weighted composition operators on S (Rd)},
  author={Asensio, Vicente and Jord{\'a}, Enrique and Kalmes, Thomas},
  journal={Journal of Functional Analysis},
  volume={288},
  number={3},
  pages={110745},
  year={2025},
  publisher={Elsevier}
}

@article{albanese2023spectral,
  title={Spectral properties of generalized Ces{\`a}ro operators in sequence spaces},
  author={Albanese, Angela A and Bonet, Jos{\'e} and Ricker, Werner J},
  journal={Revista de la Real Academia de Ciencias Exactas, F{\'\i}sicas y Naturales. Serie A. Matem{\'a}ticas},
  volume={117},
  number={4},
  pages={140},
  year={2023},
  publisher={Springer}
}

@article{albanese2022dynamics,
  title={Dynamics of composition operators on function spaces defined by local and global properties},
  author={Albanese, Angela A and Jord{\'a}, Enrique and Mele, Claudio},
  journal={J. Math. Anal. Appl},
  volume={514},
  pages={126303},
  year={2022}
}

@article{albanese2022spectra,
  title={Spectra and ergodic properties of multiplication and convolution operators on the space S (R)},
  author={Albanese, Angela A and Mele, Claudio},
  journal={Revista Matem{\'a}tica Complutense},
  volume={35},
  number={3},
  pages={739--762},
  year={2022},
  publisher={Springer}
}

@article{Blois2025,
  author    = {Artur Blois and Osmar R. Severiano},
  title     = {Dynamics for Affine Composition Operators on Weighted Bergman Spaces of a Half Plane},
  journal   = {Bulletin of the Brazilian Mathematical Society, New Series},
  year      = {2025},
  volume    = {57},
  number    = {1},
  pages     = {2},
  issn      = {1678-7714},
  doi       = {10.1007/s00574-025-00491-2},
  url       = {https://doi.org/10.1007/s00574-025-00491-2}
}

@misc{blois2026shadowing,
  title={Shadowing phenomenon for composition operators on the Hardy space $H^2(\mathbb{D})$},
  author={Blois, Artur and Eidt, Ben-Hur and Lupatini, Paulo and Severiano, Osmar R},
  year={2026},
  eprint={2603.10575},
  archivePrefix={arXiv}
}

@misc{alvarez2025spectra,
  title={Spectra of composition operators on Paley-Wiener spaces and some consequences},
  author={{\'A}lvarez, Carlos F and Severiano, OR},
  year={2025},
  eprint={2508.19975},
  archivePrefix={arXiv}
}

@article{Alvarez2025,
  author    = {Carlos F. Álvarez and Javier Henríquez-Amador},
  title     = {Dynamical Properties for Composition Operators on {$H^{2}(\mathbb{C}_{+})$}},
  journal   = {Bulletin of the Brazilian Mathematical Society, New Series},
  year      = {2025},
  volume    = {56},
  number    = {1},
  pages     = {10},
  doi       = {10.1007/s00574-024-00435-2},
  url       = {https://doi.org/10.1007/s00574-024-00435-2},
  issn      = {1678-7714}
}

@article{Kalmes2022mean,
  author  = {Thomas Kalmes and Daniel Santacreu},
  title   = {Mean ergodic composition operators on spaces of smooth functions and distributions},
  journal = {Proceedings of the American Mathematical Society},
  volume  = {150},
  number  = {6},
  pages   = {2603--2616},
  year    = {2022},
  doi     = {10.1090/proc/15894},
  url     = {https://doi.org/10.1090/proc/15894}
}
\bibliographystyle{amsplain}

\end{document}